\documentclass[a4paper,reqno,oneside]{amsart}
\usepackage{fullpage}
\usepackage{setspace}
\usepackage{datetime}

\usepackage{amsmath}
\usepackage{amsfonts}
\usepackage{amssymb}
\usepackage{verbatim,enumitem,graphics}
\usepackage{microtype}
\usepackage{xcolor}
\usepackage[backref=page]{hyperref}
\hypersetup{%
    colorlinks,
    linkcolor={red!50!black},
    citecolor={green!50!black},
    urlcolor={blue!50!black}
}
\usepackage{dsfont}

\usepackage[utf8]{inputenc}

\usepackage[abbrev,msc-links,backrefs]{amsrefs} 
\usepackage{doi}
\usepackage{fnpct}

\usepackage{enumitem}
\usepackage{MnSymbol}

\setenumerate{leftmargin=*} 

\usepackage{enumitem}

\newtheorem{theorem}{Theorem}[section]
\newtheorem{lemma}[theorem]{Lemma}
\newtheorem{proposition}[theorem]{Proposition}

\newtheorem{corollary}[theorem]{Corollary}
\newtheorem{conjecture}[theorem]{Conjecture}

\newtheorem{remark}[theorem]{Remark}

\newcommand{\oldqed}{}
\def\endofClaim{\hfill\scalebox{.6}{$\Box$}}

\newcommand{\gammaj}{\gamma^{(j)}}
\newcommand{\sigmaj}{\sigma^{(j)}}

\newcommand{\NN}{\mathbb{N}}

\newcommand{\PP}{\mathbb{P}}
\newcommand{\RR}{\mathbb{R}}

\newcommand{\EE}{\mathbb{E}}

\newcommand{\ZZ}{\mathbb{Z}}

\newcommand{\eps}{\varepsilon}

\newcommand{\cV}{\mathcal{V}}

\newcommand{\supp}{\mathrm{supp}}
\newcommand{\Msym}{M^{\mathrm{sym}}_n}
\renewcommand{\Pr}[1]{\PP\left[#1\right]}
\newcommand{\Prob}[2]{\PP_{#1}\left[#2\right]}

\newcommand{\Exp}[1]{\EE\left[#1\right]}
\newcommand{\Expu}[2]{\EE_{#1}\left[#2\right]}

\newcommand{\Qnd}{Q_{n,d}}

\def\spn{\mathrm{span}}
\def\bv{\boldsymbol{v}}
\def\ba{\boldsymbol{a}}

\def\bq{\boldsymbol{q}}
\def\bone{\boldsymbol{1}}
\def\bzero{\boldsymbol{0}}

\title{On sparse random combinatorial matrices}

\author[E. Aigner-Horev]{Elad Aigner-Horev}
\author[Y. Person]{Yury Person}

\thanks{
 YP is supported by the Carl Zeiss Foundation and by DFG grant PE 2299/3-1.  
}

\address{Department of Mathematics and Computer Science, Ariel University, Ariel, Israel}
\email{horev@ariel.ac.il}

\address{Institut f\"ur Mathematik, Technische Universit\"at Ilmenau, 98684 Ilmenau, Germany}
\email{yury.person@tu-ilmenau.de}

\begin{document}
\onehalfspace

\shortdate
\yyyymmdddate
\settimeformat{ampmtime}
\date{\today, \currenttime}
\footskip=28pt
\allowdisplaybreaks

\begin{abstract}
Let $\Qnd$ denote the random {\sl combinatorial matrix} whose rows are independent of one another and such that each row is sampled uniformly at random from the subset of vectors in $\{0,1\}^n$ having precisely $d$ entries equal to $1$. We present a short proof of the fact that $\Pr{\det(\Qnd)=0} = O\left(\frac{n^{1/2}\log^{3/2} n}{d}\right)=o(1)$, whenever $d=\omega(n^{1/2}\log^{3/2} n)$. In particular, our proof accommodates {\sl sparse} random combinatorial matrices in the sense that $d = o(n)$ is allowed. 

We also consider the singularity of deterministic integer matrices $A$ randomly perturbed by a sparse combinatorial matrix. In particular, we prove that $\Pr{\det(A+\Qnd)=0}=O\left(\frac{n^{1/2}\log^{3/2} n}{d}\right)$, again, whenever $d=\omega(n^{1/2}\log^{3/2} n)$ and $A$ has the property that $(\bone,-d)$ is not an eigenpair of $A$.  
\end{abstract}

\maketitle

\section{Introduction}
\label{sec:intro}
For an integer  $1 \leq d \leq n$, we write $Q_{n,d}$ to denote 
a random $n\times n$ matrix over $\{0,1\}$ whose rows are sampled independently and uniformly at random from the set 
$
\mathcal{S}_d := \Big\{\bv \in \{0,1\}^n: \sum_{i=1}^n \bv_i = d\Big\}
$.
The matrix $Q_{n,d}$ is often referred to as a {\em combinatorial matrix}. 
Such a matrix is said to be {\em sparse} if $d = o(n)$. 

The random matrix $Q_n := Q_{n,n/2}$ ($n$ even) has attracted much attention of late. 
Nguyen~\cite{Nguyen13} first considered its {\sl singularity probability} by showing that $\Pr{\det(Q_n)=0}=O(n^{-C})$ for any $C>0$. Subsequently, Ferber, Jain, Luh and Samotij~\cite{FJLS19} improved the singularity probability to be exponential and of the form $2^{-n^{0.1}}$. Jain~\cite{Jain2019a} extended the latter bound to the study of $s(Q_n)$; the least singular value of $Q_n$. Recently,  
Tran~\cite{Tran20} established, the asymptotically optimal bound $\Pr{s_n(Q_n)\le \frac{\eps}{\sqrt{n}}}\le C\eps+2^{-cn}$ for all $\eps\ge 0$ and some absolute constants $c$, $C>0$. In particular this implies $\PP[\det (Q_n)=0]\le 2^{-cn}$.

All of the aforementioned results extend to $d=\Omega(n)$. It is natural to ask for which further $d$, the random matrix $Q_{n,d}$ remains nonsingular {\sl asymptotically almost surely} (a.a.s.\ , hereafter)\footnote{That is, $\Pr{\det(Q_{n,d})=0} \to 0$ as $n \to \infty$.}. It is easy to see that, for  $d\le (1-\eps)\log n$, the matrix $\Qnd$ a.a.s.\ contains a zero column. We conjecture as follows.  

\begin{conjecture}\label{conj:AHP}
For every $\eps>0$ and $d\ge (1+\eps)\log n$, we have
\[
\Pr{\det(\Qnd)=0}=o(1).
\]
\end{conjecture}

In this note, we provide a short proof of the following result which holds for a significantly wider range of $d$ from that seen in the aforementioned results. Let $\omega(t)$ denote any function $f(t)$ for which $f(t)/t$ tends to infinity arbitrarily slowly as $t\rightarrow \infty$.

\begin{theorem}\label{thm:Q_n-a} For $d=\omega(n^{1/2}\log^{3/2} n)$ we have
\[
\PP[\det(\Qnd)=0] = O\left(\frac{n^{1/2}\log^{3/2} n}{d}\right)=o(1).
\]
\end{theorem}

The probability bounds stated in Theorem~\ref{thm:Q_n-a} are probably far from optimal (in view of the aforementioned works~\cite{FJLS19,Jain2019a,Tran20}). 
Our proof is inspired by the recent exceedingly short proof by Ferber~\cite{Ferber20} that $\Pr{\det(\Msym)=0}=O\left(\frac{\log^C n}{n^{1/2}}\right)$, where matrix $\Msym$ denotes the random symmetric $\pm1$-matrix; the latter being one of the most studied models of random discrete matrices with dependencies among the entries. Our proof is also influenced by the arguments of Ferber, Jain, Luh and Samotij~\cite{FJLS19}; in our account, though, we avoid appealing to the so called {\sl counting theorem} associated with the inverse Littlewood-Offord problem established in~\cite{FJLS19}; this in turn renders our argument somewhat shorter and extendable to $d=\omega(n^{1/2}\log^{3/2} n)$.

We also address the nonsingularity of deterministic matrices with integer entries perturbed by a sparse random combinatorial matrix. 

\begin{theorem}\label{thm:Q_n-b}
Let $A$ be an $n\times n$ matrix with integer entries 
such that $(\bone, -d)$ is not an eigenpair of $A$, where $\bone$ denotes the all-ones vector.
Then $d=\omega(n^{1/2}\log^{3/2} n)$ we have 
\[
\PP[\det(A+\Qnd)=0]=O\left(\frac{n^{1/2}\log^{3/2} n}{d}\right)=o(1).
\]
\end{theorem}

Similar perturbations problems by other models of random matrices and for a wider class of deterministic matrices were extensively studied in the literature as well. We refer the reader to the works of 
Sankar, Spielman, and Teng~\cite{SST06}, Tao and Vu~\cite{TV08,TV10}, Jain~\cite{Jain2019b}, Livshyts, Tikhomirov, and Vershynin~\cite{LTV19}, and Jain, Sah, and Sawhney~\cite{JSS20a}. For an excellent account regarding {\sl combinatorial random matrix theory} in general, see the recent survey by Vu~\cite{Vu20}. 

\subsection{Proof overview} The proof of Theorems~\ref{thm:Q_n-a} and~\ref{thm:Q_n-b} proceeds by reducing the problem to the finite field setting and working over $\ZZ_p$. It then turns out to be sufficient to consider the following two cases:  almost constant vectors; i.e.\ those that have the same entries in all but $d/(10\log n)$ positions (in particular vectors with small support), and vectors which are not almost constant. 

The former case can be dealt with rather straightforward union bound, while the latter case is studied through a Hal\'asz-type argument~\cite{Halasz77} for the Littlewood-Offord problem. Then, a first moment argument yields that with probability $o(1)$ the kernel of the matrix $\Qnd$ (over $\ZZ_p$) is less than $p$ and hence the matrix is nonsingular.

\section{Invertibility with respect to a single vector}

For $n \in \mathbb{N}$ and a prime $p$, write $\mathbb{Z}^{n\times n}$ to denote the set of $n \times n$-matrices over the integers; write $\mathbb{Z}^{n \times n}_p$ to denote those matrices whose entries are taken from $\mathbb{Z}_p$. Given $A \in \mathbb{Z}^{n \times n}$, let $A_p$ denote the matrix obtained from $A$ by reducing the elements of $A$ modulo $p$. For a vector $\bv \in \ZZ^n_p$, write $\supp(\bv) : = \{i\colon \bv_i\not\equiv 0 \pmod p\}$ to denote the {\em support} of $\bv$.
 We start with the following simple observation allowing us to carry out our arguments modulo a prime $p$. 

\begin{proposition}\label{prop:reduce_Zp}
Let $p$ be prime and let $A\in\ZZ^{n\times n}$. Then,
\[
\PP[\det(A+\Qnd)=0]\le \PP[\det(A_p+\Qnd)\equiv 0 \pmod p].
\]
\end{proposition}

\begin{proof}
If $\det(A+\Qnd)=0$ holds, then $\det(A+\Qnd)\equiv 0 \pmod p$, and hence $\det(A_p+\Qnd) \equiv 0 \pmod p$.
\end{proof}

Let $\cV$ denote the set of vectors $\bv\in\ZZ_p^n$ satisfying $|\{i\colon \bv_i=b\}|\le n-d/(10\log n)$ for every $b\in\ZZ_p$, i.e.\ the {\sl level} sets of each entry are not too big. The following lemma, which is a counterpart of~\cite[Lemma~2.2]{Ferber20},  asserts that for $\bv \in \cV$, the random vector $\Qnd \bv$ is essentially uniformly distributed. 
It is in this lemma that we incur a lower bound on $d$. 

\begin{lemma}\label{lem:cV}
Let  $d=\omega(n^{1/2}\log^{3/2} n)$
and let $p=\Theta\left(\frac{d}{n^{1/2}\log^{3/2} n}\right)$ be a prime.  Let $\bv\in \cV$ and $\ba\in\ZZ_p^n$. Then,
\[
\Pr{\Qnd \bv=\ba}=\frac{1+o(1)}{p^n}.
\]
\end{lemma}

\begin{proof}
We write $e_p(x):=\exp\left(\frac{2\pi ix}{p}\right)$. Then,  
$$
\Pr{\Qnd \bv=\ba}=\Exp{\delta_0(\Qnd \bv-\ba)}=\prod_{j=1}^n \Exp{\delta_0((\Qnd \bv-\ba)_j)},
$$ 
where here $\delta_0(x) = 1$ if $x=0$ and zero otherwise. 
Writing $\Qnd(j)$ to denote the $j$-th row of $\Qnd$, we reach 
\begin{align}
\Pr{\Qnd \bv=\ba}&=\prod_{j=1}^n\Expu{\Qnd(j)}{\delta_0(\Qnd(j)^{\intercal}\bv-\ba_j)} \nonumber \\ 
&=\prod_{j=1}^n\Expu{\Qnd(j)}{\Expu{\xi_j\sim\ZZ_p}{e_p(\xi_j(\Qnd(j)^{\intercal}
\bv-\ba_j))}} \nonumber \\
& =\prod_{j=1}^n\Expu{\Qnd(j)}{\frac{1}{p}\sum_{\xi_j\in\ZZ_p}{e_p(\xi_j(\Qnd(j)^{\intercal}\bv-\ba_j))}} \nonumber \\
&=\prod_{j=1}^n\left(\frac{1}{p}+\Expu{\Qnd(j)}{\frac{1}{p}\sum_{\xi_j\in\ZZ_p\setminus\{0\}}{e_p(\xi_j(\Qnd(j)^{\intercal}\bv-\ba_j))}}\right) \nonumber \\
& =\prod_{j=1}^n\left(\frac{1}{p}+ \frac{1}{p}\sum_{\xi_j\in\ZZ_p\setminus\{0\}}\Expu{\Qnd(j)}{e_p(\xi_j(\Qnd(j)^{\intercal}\bv-\ba_j))}\right) \nonumber\\
& =\frac{1}{p^n}+\frac{1}{p^n}\sum_{\xi\in\ZZ_p^n\setminus\{0\}}\left({\prod_{j=1}^n \Expu{\Qnd(j)}{e_p(\xi_j(\Qnd(j)^{\intercal}\bv-\ba_j))}}\right).\label{eq:prob-est}
\end{align}

We pursue an estimate for $\left|\sum_{\xi\in\ZZ_p^n\setminus\{0\}}\left({\prod_{j=1}^n\Expu{\Qnd(j)}{e_p(\xi_j(\Qnd(j)^{\intercal}\bv-\ba_j))}}\right)\right|$. 

\begin{align}
\left|\sum_{\xi\in\ZZ_p^n\setminus\{0\}}\left({\prod_{j=1}^n\Expu{\Qnd(j)}{e_p(\xi_j(\Qnd(j)^{\intercal}\bv-\ba_j))}}\right)\right| \nonumber 
& \le \sum_{\xi\in\ZZ_p^n\setminus\{0\}}\prod_{j=1}^n \left|\Expu{\Qnd(j)}{e_p(\xi_j(\Qnd(j)^{\intercal}\bv-\ba_j))}\right| \nonumber \\
& =\sum_{\xi\in\ZZ_p^n\setminus\{0\}}\prod_{j=1}^n \left|\Expu{\Qnd(j)}{e_p(\xi_j \Qnd(j)^{\intercal}\bv)}\right|\nonumber \\
&=\sum_{\xi\in\ZZ_p^n\setminus\{0\}}\prod_{j\in\supp(\xi)} \left|\Expu{\Qnd(j)}{e_p(\xi_j \Qnd(j)^{\intercal}\bv)}\right| \nonumber \\
&=\sum_{\xi\in\ZZ_p^n\setminus\{0\}}\prod_{j\in\supp(\xi)} \left|\Expu{\Qnd(j)}{e_p\left(\xi_j\sum_{\ell\in\supp(\bv)} \Qnd(j,\ell)\bv_{\ell}\right)}\right|.\label{eq:error-est-a}
\end{align}

The set $T_j:=\{\ell\colon Q_n(j,\ell)=1\}$ is a $d$-element set sampled uniformly from $\binom{[n]}{d}$ of $[n]$. An alternative generation of $T_j$ reads as follows. For any fixed $\gamma\in\{0,1\}^{d}$ and any permutation $\sigma\in S_n$ of $[n]$ sampled uniformly at random from $S_n$, set
$$
T_{\gamma,\sigma}:=\bigcup_{i\in[d]}\Big(\{\sigma(i)\colon \gamma_i=1\}\cup\{\sigma(i+d)\colon \gamma_i=0\}\Big).
$$
Then, for every set $T\in\binom{[n]}{d}$, the equality $\PP[T_{\gamma,\sigma} = T] = \frac{d!(n-d)!}{n!} = \binom{n}{d}^{-1}$ holds. Indeed, the number of ways to map the members of $T$ to the set of positions $\cup_{i\in[d]}\left(\{i\colon \gamma_i=1\}\cup\{i+d\colon \gamma_i=0\}\right)$ is $d!$. There are $(n-d)!$ ways to complete the mapping from $[n]$ to the remaining set of positions. Any permutation of $[n]$ thus formed gives rise to $T$ through $T_{\gamma,\sigma}$. 
It follows that for any fixed $\gamma \in \{0,1\}^d$, the set $T_{\gamma,\sigma}$ is uniformly distributed in $\binom{[n]}{d}$; this remains so if $\gamma\sim\{0,1\}^d$ itself is a sequence of uniformly distributed Bernoulli variables. Such a coupling (for $d=n/2$) is used e.g.\ in the proof of~\cite[Proposition~4.10]{LLTTJY17} and also in~\cite{FJLS19}.

In what follows, we view each set $T_j$ as being generated by a random permutation $\sigma^{(j)}$ and a uniformly random sequence $\gamma^{(j)}\in\{0,1\}^d$. 
We continue with~\eqref{eq:error-est-a}.

\begin{align}
\sum_{\xi\in\ZZ_p^n\setminus\{0\}} & \prod_{j\in\supp(\xi)} \left|\Expu{Q_n(j)}{e_p\left(\xi_j\sum_{\ell\in\supp(\bv)} Q_n(j,\ell)\bv_{\ell}\right)}\right|  = \sum_{\xi\in\ZZ_p^n\setminus\{0\}}\prod_{j\in\supp(\xi)} \left|\Expu{T_j}{e_p\left(\xi_j\sum_{\ell\in\supp(\bv)\cap T_j}\bv_{\ell}\right)}\right| \nonumber \\
& =\sum_{\xi\in\ZZ_p^n\setminus\{0\}}\prod_{j\in\supp(\xi)} 
\left|\Expu{\sigmaj, \gammaj}{e_p\left(\xi_j\left(\sum_{\ell=1}^{d} \gammaj_\ell \bv_{\sigmaj(\ell)}+\sum_{s=1}^{d} (1-\gammaj_s)\bv_{\sigmaj(s+d)}\right)\right)}\right| \nonumber \\
& =\sum_{\xi\in\ZZ_p^n\setminus\{0\}}\prod_{j\in\supp(\xi)} \left|\Expu{\sigmaj}{\Expu{\gammaj}{e_p\left(\xi_j\left(\sum_{\ell=1}^{d} \gammaj_\ell (\bv_{\sigmaj(\ell)}-\bv_{\sigmaj(\ell+d)})+\sum_{s=1}^{d} \bv_{\sigmaj(s+d)}\right)\right)}}\right|\nonumber \\
& =\sum_{\xi\in\ZZ_p^n\setminus\{0\}}\prod_{j\in\supp(\xi)} \left|\Expu{\sigmaj}{e_p\left(\xi\sum_{s=1}^{d} \bv_{\sigmaj(s+d)}\right)\Expu{\gammaj}{e_p\left(\xi_j\sum_{\ell=1}^{d} \gammaj_\ell (\bv_{\sigmaj(\ell)}-\bv_{\sigmaj(\ell+d)})\right)}}\right| \nonumber \\
& =\sum_{\xi\in\ZZ_p^n\setminus\{0\}}\prod_{j\in\supp(\xi)} \left|\Expu{\sigmaj}{e_p\left(\xi_j\sum_{s=1}^{d} \bv_{\sigmaj(s+d)}\right)\prod_{\ell=1}^{d}\Expu{\gammaj}{e_p\left(\xi_j \gammaj_\ell (\bv_{\sigmaj(\ell)}-\bv_{\sigmaj(\ell+d)})\right)}}\right| \nonumber \\
& =\sum_{\xi\in\ZZ_p^n\setminus\{0\}}\prod_{j\in\supp(\xi)} \left|\Expu{\sigmaj}{e_p\left(\xi_j\sum_{s=1}^{d} \bv_{\sigmaj(s+d)}\right)\prod_{\ell=1}^{d}\left(\frac{1+e_p(\xi_j(\bv_{\sigmaj(\ell)}-\bv_{\sigmaj(\ell+d)}))}{2}\right)}\right| \nonumber \\
&\le \sum_{\xi\in\ZZ_p^n\setminus\{0\}}\prod_{j\in\supp(\xi)} \Expu{\sigmaj}{\prod_{\ell=1}^{d}\left|\frac{1+e_p(\xi_j(\bv_{\sigmaj(\ell)}-\bv_{\sigmaj(\ell+d)}))}{2}\right|}.\label{eq:error-est-b}
\end{align}

Since 
$$
\left|\frac{1+e_p(\xi_j(\bv_{\sigmaj(\ell)}-\bv_{\sigmaj(\ell+d)}))}{2}\right|=|\cos\left(\pi \xi_j(\bv_{\sigmaj(\ell)}-\bv_{\sigmaj(\ell+d)})/p\right)|,
$$
then owing to the inequality $|\cos (\pi m/p)|  \leq e^{-2/p^2}$ that holds whenever $m \in \ZZ_p \setminus \{0\}$, it follows that if $\bv_{\sigmaj(\ell)}-\bv_{\sigmaj(\ell+d)}\neq 0$, then we have the following estimate (as $\xi_j\neq 0$),
\[
|\cos\left(\pi \xi_j(\bv_{\sigmaj(\ell)}-\bv_{\sigmaj(\ell+d)})/p\right)|\le e^{-2/p^2}.
\]

Returning to~\eqref{eq:error-est-b}, we may now write  

\begin{align}
\sum_{\xi\in\ZZ_p^n\setminus\{0\}}\prod_{j\in\supp(\xi)} & \Expu{\sigmaj}{\prod_{\ell=1}^{d}\left|\frac{1+e_p(\xi_j(\bv_{\sigmaj(\ell)}-\bv_{\sigmaj(\ell+d)}))}{2}\right|} \nonumber \\
& \le \sum_{\xi\in\ZZ_p^n\setminus\{0\}}\prod_{j\in\supp(\xi)} \Expu{\sigmaj}{\prod_{\ell\colon \bv_{\sigmaj(\ell)}\neq \bv_{\sigmaj(\ell+d)}}e^{-2/p^2}}. \label{eq:exp-cos}
\end{align}

By assumption $\bv\in\cV$. Hence, the number of (ordered) pairs $(s,t)$ such that $\bv_s-\bv_t\neq 0$ is at least 
$$
\frac{d}{10\log n} \left(n-\frac{d}{10\log n}\right)>\frac{dn}{11\log n}.
$$ 
Let $X_j$ denote the number of pairs $(\sigmaj(\ell),\sigmaj(\ell+d))$ satisfying $\bv_{\sigmaj(\ell)} - \bv_{\sigmaj(\ell+d)} \neq 0$ for $\ell
\in[d/(20\log n)]$. Then, $\Exp{X_j} \geq d^2/(220 n\log^2 n)$. 
We can view value of $X_j$ evolving as pairs $(\sigmaj(\ell),\sigmaj(\ell+d))$ are revealed and 
the probability that $\bv_{\sigmaj(\ell)} - \bv_{\sigmaj(\ell+d)} \neq 0$ (conditioned on the first $\ell-1$ pairs) for any $
\ell\le d/(20\log n )$ is at least $\Omega(d/(n\log n))$. It follows by a standard submartingale inequality (see e.g.~\cite[Lemma~2.2]{ABHKP14b}) that 
\[
\Pr{X_j< \EE\,X_j/2}\le \exp\left(-\EE\,X_j/12\right).
\]
Thus, with probability at most $\exp\left(-\frac{d^2}{12\cdot 220\cdot n\log^2 n}\right)$ we have $X_j\le \EE\, X_j/2$ so that 
\begin{align}
\Expu{\sigmaj}{\prod_{\ell\colon \bv_{\sigmaj(\ell)}\neq \bv_{\sigmaj(\ell+d)}}e^{-2/p^2}}\nonumber 
& = \Expu{\sigmaj}{e^{-2X_j/p^2}} \nonumber \\
&\le \exp\left(-\frac{d^2}{2^{12}\cdot n\log^2 n}\right)+\exp\left(-\frac{d^2}{2^{12}\cdot p^2\cdot n\log^2 n }\right)\nonumber \\
& \le 2\exp\left(- C\log n\right),\label{eq:expect-est}
\end{align}
where here the third inequality is owing to $d=\omega(n^{1/2}\log^{3/2} n)$ and $p=\Theta\left(\frac{d}{n^{1/2}\log^{3/2} n}\right)$.

We conclude with the following estimation for $\Pr{\Qnd\cdot \bv=\ba}$ appearing in~\eqref{eq:prob-est} using~\eqref{eq:error-est-a},~\eqref{eq:error-est-b},~\eqref{eq:exp-cos} and~\eqref{eq:expect-est}.

\begin{align}
\left\mid\Pr{\Qnd \bv= \ba}-\frac{1}{p^n}\right\mid & \overset{\eqref{eq:prob-est}}{\le}\frac{1}{p^n} \left|\sum_{\xi\in\ZZ_p^n\setminus\{0\}}\left({\prod_{j=1}^n\Expu{Q_n(j)}{e_p(\xi_j(Q_n(j)^{\intercal}\bv-\ba_j))}}\right)\right| \nonumber \\
& \overset{\eqref{eq:error-est-a}}{\le} \frac{1}{p^n}\sum_{\xi\in\ZZ_p^n\setminus\{0\}}\prod_{j\in\supp(\xi)} \left|\Expu{Q_n(j)}{e_p\left(\xi_j\sum_{\ell\in\supp(\bv)} Q_n(j,\ell)\bv_{\ell}\right)}\right| \nonumber \\
& \overset{\eqref{eq:error-est-b}}{\le} \frac{1}{p^n}\sum_{\xi\in\ZZ_p^n\setminus\{0\}}\prod_{j\in\supp(\xi)} \Expu{\sigmaj}{\prod_{\ell=1}^{d}\left|\frac{1+e_p(\xi_j(\bv_{\sigma(\ell)}-\bv_{\sigma(\ell+d)}))}{2}\right|} \nonumber \\
& \overset{\eqref{eq:exp-cos}}{\le} \frac{1}{p^n}\sum_{\xi\in\ZZ_p^n\setminus\{0\}}\prod_{j\in\supp(\xi)} \Expu{\sigmaj}{\prod_{\ell\colon \bv_{\sigmaj(\ell)}\neq \bv_{\sigmaj(\ell+d)}}e^{-2/p^2}} \nonumber \\
& \overset{\eqref{eq:expect-est}}{\le} \frac{1}{p^n}\sum_{\xi\in\ZZ_p^n\setminus\{0\}}\prod_{j\in\supp(\xi)} 2\exp\left(-C\log n\right) \nonumber \\
& \overset{\phantom{\eqref{eq:expect-est}}}{=}
\frac{1}{p^n}\sum_{\xi\in\ZZ_p^n\setminus\{0\}} 2^{|\supp(\xi)|}\exp\left(-C|\supp(\xi)|\log n\right).\label{eq:final-sum}
\end{align}

For the final sum appearing in~\eqref{eq:final-sum} we observe that
\[
\frac{1}{p^n}\sum_{\xi\in\ZZ_p^n\setminus\{0\}} 2^{|\supp(\xi)|}\exp\left(-C|\supp(\xi)|\log n\right)= \frac{1}{p^n}
\sum_{s=1}^{n}\binom{n}{s}p^s2^s\exp\left(-Cs\log n\right)=\frac{o(1)}{p^n}.
\]
It follows that
\[
\left\mid\Pr{\Qnd\cdot \bv=\ba}-\frac{1}{p^n}\right\mid \le\frac{o(1)}{p^n}
\]
concluding the proof.
\end{proof}

Let $\bone_{\supp(\bv)} \in \ZZ^n_p$ denote the vector attained from $\bv$ by setting the $i$th entry to $1$ for every $i \in \supp(\bv)$ and zero otherwise. 
For a set $S \subseteq [n]$, we write $\bone_{S}$ to denote the {\em characteristic} vector of $S$. Finally, set $\spn(\bone) := \{\alpha \cdot \bone: \alpha \in \ZZ_p\}$ and 
$
\bone_{\supp(v)} \cdot \ZZ_p := \{\alpha \cdot \bone_{\supp(\bv)}: \alpha \in \ZZ_p\}
$.

Proposition~\ref{prop:LO-light}, stated next, is our counterpart of~\cite[Observation~2.1]{Ferber20}; the lack of independence between entries in our setting renders our argument somewhat more involved. 
We provide two proofs of Proposition~\ref{prop:LO-light}. The lengthier one that is more combinatorial in spirit, so to speak, we postpone until the Appendix below. 

\begin{proposition}\label{prop:LO-light} Let $p$ be a prime. 
Let $A\in\ZZ_p^{n\times n}$ and $\bv\in\ZZ_p^n\setminus \spn(\bone)$. Then, there exists  a constant $c'>0$ such that
\[
\PP[(A+\Qnd)\bv=0]\le e^{-c'd}
\]
holds for all $d\in [1,n/2]\cap \NN$. 
\end{proposition}
\begin{proof}
Let $\bq$ be a vector sampled uniformly at random from $\mathcal{S}_d$. We show that for \emph{any} $b\in \ZZ_p$, 
\begin{equation}\label{eq:=b}
\PP[\bq^{\intercal}\bv=b]\le 1-c'\frac{d}{n}
\end{equation}
 for some absolute constant $c'>0$. Assuming~\eqref{eq:=b} and utilising the independence of the rows of $\Qnd$, we may write 
$$
\PP[(A+\Qnd)\bv=0]=\PP[\Qnd \bv= -Av]=\prod_{i=1}^n \PP[( \Qnd \bv)_i=(-A\bv)_i]\overset{\eqref{eq:=b}}{\le} \left(1-c'\frac{d}{n}\right)^n\le e^{-c'd};
$$
establishing the claim.  
It remains to prove~\eqref{eq:=b}.

Employing a similar coupling to that seen just below~\eqref{eq:error-est-a}, observe that a set $T\in\binom{[n]}{d}$, chosen uniformly at random, can be generated using a pair consisting of a permutation of $[n]$, namely $\sigma$, chosen uniformly at random, and a sequence $\gamma\sim\{\pm 1\}^d$ of independent {\sl Rademacher} random variables, by setting    
\[
T_{\gamma,\sigma}:=\bigcup_{i\in[d]}\Big(\{\sigma(i)\colon (1+\gamma_i)/2=1\}\cup\{\sigma(i+d)\colon (1-\gamma_i)/2=1\}\Big).
\]
As before, the latter distributes uniformly over $\binom{[n]}{d}$. 

We may now write
\begin{align}
\Pr{\bq^{\intercal}\bv=b} & = \Expu{\bq}{\delta_0(\bq^{\intercal}\bv-b)} \nonumber\\
&=\Expu{\bq}{\Expu{\xi\sim\ZZ_p}{e_p(\xi(\bq^{\intercal}\bv-b))}} \nonumber \\
& =\Expu{T}{\Expu{\xi\sim\ZZ_p}{e_p(\xi(\sum_{\ell\in T} \bq_\ell \bv_\ell-b))}} \nonumber \\
& =\Expu{\sigma,\gamma}{\Expu{\xi\sim\ZZ_p}{e_p\left(\xi\left(\sum_{s=1}^d \frac{1+\gamma_s}{2}v_{\sigma(s)}+\sum_{\ell=1}^d \frac{1-\gamma_\ell}{2}v_{\sigma(\ell+d)}-b\right)\right)}}.
\label{eq:single-row-est}
\end{align}

Rearranging the terms in $e_p(\cdot)$ in the last line of~\eqref{eq:single-row-est} yields
\begin{align}
\mathbb{E}_{\sigma,\gamma}\bigg[&\Expu{\xi\sim\ZZ_p}{e_p\left(\xi\left(\sum_{s=1}^d \frac{1+\gamma_s}{2}v_{\sigma(s)}+\sum_{\ell=1}^d \frac{1-\gamma_\ell}{2}v_{\sigma(\ell+d)}-b\right)\right)}\bigg]\nonumber\\
&=\Expu{\sigma,\gamma}{\Expu{\xi\sim\ZZ_p}{e_p\left(\xi\left(\sum_{\ell=1}^d \frac{\gamma_\ell}{2}(v_{\sigma(\ell)}-v_{\sigma(\ell+d)})+\sum_{s=1}^{2d} \frac{v_{\sigma(s)}}{2}-b\right)\right)}}\nonumber\\
&=\Expu{\sigma,\gamma}{\Expu{\xi\sim\ZZ_p}{e_p\left(\xi\left(\sum_{\ell=1}^d \gamma_\ell(v_{\sigma(\ell)}-v_{\sigma(\ell+d)})+\sum_{s=1}^{2d} v_{\sigma(s)}-2b\right)\right)}}\nonumber\\
&=\EE_\sigma\Expu{\gamma}{\delta_0\left(\sum_{\ell=1}^d \gamma_\ell(v_{\sigma(\ell)}-v_{\sigma(\ell+d)})+\sum_{s=1}^{2d} v_{\sigma(s)}-2b\right)}\nonumber\\
&=\Expu{\sigma}{\Prob{\gamma}{\sum_{\ell=1}^d \gamma_\ell(v_{\sigma(\ell)}-v_{\sigma(\ell+d)})=2b-\sum_{s=1}^{2d} v_{\sigma(s)}}}.
\label{eq:single-row-est-b}
\end{align}
Now, since $v\not\in\spn(\bone)$, observe that for a fixed permutation $\sigma$ with at least one pair $(v_{\sigma(\ell)},v_{\sigma(\ell+d)})$ such that $v_{\sigma(\ell)}\neq v_{\sigma(\ell+d)}$ for some $\ell\in[d]$, we have
\[
\Prob{\gamma}{\sum_{\ell=1}^d \gamma_\ell(v_{\sigma(\ell)}-v_{\sigma(\ell+d)})=2b-\sum_{s=1}^{2d} v_{\sigma(s)}}\le 1/2,
\]
whereas for other permutations $\sigma$ we take the trivial upper bound 
\[
\Prob{\gamma}{\sum_{\ell=1}^d \gamma_\ell(v_{\sigma(\ell)}-v_{\sigma(\ell+d)})=2b-\sum_{s=1}^{2d} v_{\sigma(s)}}\le 1.
\]
The probability that the random permutation $\sigma$ satisfies $v_{\sigma(\ell)}\neq v_{\sigma(\ell+d)}$ for some $\ell\in[d]$ is 
at least $\frac{c d}{n}$ for some absolute constant $c>0$ (this can be proved using an argument similar to that seen in~\cite[Lemma~5.1]{FJLS19}). Hence,
\begin{equation}\label{eq:sparse-last}
\Pr{\bq^{\intercal}\bv=b}\le \frac{cd}{n}\cdot \frac{1}{2}+\left(1-\frac{cd}{n}\right)=1-\frac{cd}{2n},
\end{equation}
holds which establishes~\eqref{eq:=b} and thus completing the proof.
\end{proof}

\begin{remark}
The constant $c'$ in Proposition~\ref{prop:LO-light} can be taken to be $1/2$. 
It is also possible to provide an estimate $\Pr{\bq^{\intercal}\bv=b}\le 1-\frac{d}{2n}$ directly by checking various cases. This is more clearly seen in our alternative proof of Proposition~\ref{prop:LO-light} presented in the Appendix. 
\end{remark}

\begin{corollary}\label{cor:LO-light} Let $p$ be a prime and let $d\in[1,n/2]$ be an integer such that $p\nmid d$.
Let $\bv\in\ZZ_p^n\setminus \{0\}$. Then, 
\[
\PP[\Qnd \bv=0]\le e^{-d/2}.
\]
\end{corollary}

\begin{proof}
If $\bv \notin \spn(\bone)$, then the claim follows by Proposition~\ref{prop:LO-light}. Suppose then that $\bv \in \spn(\bone)$ and that it lies in the kernel of $\Qnd$. Then, $\bone$ lies in the kernel of $\Qnd$. Note, however, that 
$\Qnd\cdot \bone =d\cdot \bone$ and, since $p\nmid d$, it follows that $\Qnd\cdot \bone \neq 0$; a contradiction.   
\end{proof}

\section{Proofs of Theorems~\ref{thm:Q_n-a} and~\ref{thm:Q_n-b}}
\begin{proof}[Proof of Theorem~\ref{thm:Q_n-a}]
By Proposition~\ref{prop:reduce_Zp}, it suffices to prove the result over $\ZZ_p$ for a prime $p=\Theta\left(\frac{d}{n^{1/2}\log^{3/2} n}\right)$ such that $p\nmid d$. 
Similarly to~\cite{Ferber20}, define $K:=|\ker_{\ZZ_p}(\Qnd)|$. Then, 
\[
\EE\,K=\sum_{\bv\in\cV}\Pr{\Qnd\cdot \bv=0}+\sum_{\bv\not\in\cV}\Pr{\Qnd\cdot \bv=0}\le \frac{1+o(1)}{p^n} p^n+1+\binom{n}{d/(10\log n)}p^{d/(10\log n)+1}e^{-d/2}=2+o(1),
\]
where here we apply Corollary~\ref{cor:LO-light} to vectors $\bv\not\in\cV\cup\{\bzero\}$ and Lemma~\ref{lem:cV} to vectors $\bv\in\cV$.
As
\[
\Pr{\det(\Qnd)=0}\le \Pr{K\ge p}\le (2+o(1))/p,
\]
the statement follows.
\end{proof}

\begin{proof}[Proof of Theorem~\ref{thm:Q_n-b}]
Our proof of Theorem~\ref{thm:Q_n-b} is exactly as that seen for Theorem~\ref{thm:Q_n-a}, with the exception that Proposition~\ref{prop:LO-light} is used instead of Corollary~\ref{cor:LO-light}. In this manner,  we can  conclude that, with probability at most $O\left(\frac{n^{1/2}\log^{3/2} n}{d}\right)$, the kernel $\ker_{\ZZ_p}(A_p+\Qnd)$ contains $0$ and possibly $\spn(\bone)$. However, due to our assumption on $A$, i.e., that $A\cdot \bone \neq -d\cdot \bone$, the vector $\bone$ is not in $\ker_{\RR}(A+\Qnd)$. Hence, with probability $O\left(\frac{n^{1/2}\log^{3/2} n}{d}\right)$, we have $\ker_{\RR}(A+\Qnd)=\{0\}$ and the claim follows.
\end{proof}

\bibliographystyle{amsplain_yk}
\bibliography{Q_n-matrices}

\begin{bibdiv}
\begin{biblist}

\bib{ABHKP14b}{article}{
      author={Allen, P.},
      author={B\"ottcher, J.},
      author={H{\`a}n, H.},
      author={Kohayakawa, Y.},
      author={Person, Y.},
       title={Blow-up lemmas for sparse graphs},
        date={2016},
      eprint={1612.00622},
}

\bib{Ferber20}{article}{
      author={Ferber, Asaf},
       title={Singularity of random symmetric matrices--simple proof},
        date={2020},
     journal={arXiv preprint arXiv:2006.07439},
}

\bib{FJLS19}{article}{
      author={Ferber, Asaf},
      author={Jain, Vishesh},
      author={Luh, Kyle},
      author={Samotij, Wojciech},
       title={On the counting problem in inverse {L}ittlewood--{O}fford
  theory},
        date={2019},
     journal={arXiv preprint arXiv:1904.10425},
}

\bib{Halasz77}{article}{
      author={Hal{\'a}sz, G{\'a}bor},
       title={Estimates for the concentration function of combinatorial number
  theory and probability},
        date={1977},
     journal={Periodica Mathematica Hungarica},
      volume={8},
      number={3-4},
       pages={197\ndash 211},
}

\bib{Jain2019a}{article}{
      author={Jain, Vishesh},
       title={Approximate {S}pielman-{T}eng theorems for the least singular
  value of random combinatorial matrices},
        date={2019},
     journal={arXiv preprint arXiv:1904.10592},
}

\bib{Jain2019b}{article}{
      author={Jain, Vishesh},
       title={Quantitative invertibility of random matrices: a combinatorial
  perspective},
        date={2019},
     journal={arXiv},
       pages={arXiv\ndash 1908},
}

\bib{JSS20a}{article}{
      author={Jain, Vishesh},
      author={Sah, Ashwin},
      author={Sawhney, Mehtaab},
       title={On the smoothed analysis of the smallest singular value with
  discrete noise},
        date={2020},
     journal={arXiv preprint arXiv:2009.01699},
}

\bib{LLTTJY17}{article}{
      author={{Litvak}, Alexander~E.},
      author={{Lytova}, Anna},
      author={{Tikhomirov}, Konstantin},
      author={{Tomczak-Jaegermann}, Nicole},
      author={{Youssef}, Pierre},
       title={{Adjacency matrices of random digraphs: singularity and
  anti-concentration}},
        date={2017},
        ISSN={0022-247X},
     journal={{J. Math. Anal. Appl.}},
      volume={445},
      number={2},
       pages={1447\ndash 1491},
}

\bib{LTV19}{article}{
      author={Livshyts, Galyna~V.},
      author={Tikhomirov, Konstantin},
      author={Vershynin, Roman},
       title={The smallest singular value of inhomogeneous square random
  matrices},
        date={2019},
     journal={arXiv preprint arXiv:1909.04219},
}

\bib{Nguyen13}{article}{
      author={{Nguyen}, Hoi~H.},
       title={{On the singularity of random combinatorial matrices}},
        date={2013},
        ISSN={0895-4801; 1095-7146/e},
     journal={{SIAM J. Discrete Math.}},
      volume={27},
      number={1},
       pages={447\ndash 458},
}

\bib{SST06}{article}{
      author={{Sankar}, Arvind},
      author={{Spielman}, Daniel~A.},
      author={{Teng}, Shang-Hua},
       title={{Smoothed analysis of the condition numbers and growth factors of
  matrices}},
        date={2006},
        ISSN={0895-4798; 1095-7162/e},
     journal={{SIAM J. Matrix Anal. Appl.}},
      volume={28},
      number={2},
       pages={446\ndash 476},
}

\bib{TV08}{article}{
      author={{Tao}, Terence},
      author={{Van}, Vu},
       title={{Random matrices: the circular law}},
        date={2008},
        ISSN={0219-1997; 1793-6683/e},
     journal={{Commun. Contemp. Math.}},
      volume={10},
      number={2},
       pages={261\ndash 307},
}

\bib{TV10}{article}{
      author={{Tao}, Terence},
      author={{Vu}, Van},
       title={{Smooth analysis of the condition number and the least singular
  value}},
        date={2010},
        ISSN={0025-5718; 1088-6842/e},
     journal={{Math. Comput.}},
      volume={79},
      number={272},
       pages={2333\ndash 2352},
}

\bib{Tran20}{article}{
      author={Tran, Tuan},
       title={The smallest singular value of random combinatorial matrices},
        date={2020},
     journal={arXiv preprint arXiv:2007.06318},
}

\bib{Vu20}{article}{
      author={Vu, Van},
       title={Recent progress in combinatorial random matrix theory},
        date={2020},
     journal={arXiv preprint arXiv:2005.02797},
}

\end{biblist}
\end{bibdiv}

\section{Appendix}

\begin{proof}[Combinatorial proof of Proposition~\ref{prop:LO-light}]
Here we provide a combinatorial (i.e.\ case by case) proof of~\eqref{eq:=b}. 

\noindent
\emph{Case 1: $|\supp(\bv)|=1$.} W.l.o.g.\ we assume that $\bv_1\neq 0$. 
As, in this case, $\bq^{\intercal}\bv = \bq_1 \bv_1 \in \{0,\bv_1\}$, we have that if $b \notin \{0,\bv_1\}$, then $\PP[\bq^{\intercal}\bv=b]=0$ holds trivialy. 
If $b = \bv_1$, then 
\[
\PP[\bq^{\intercal}\bv=b]=\PP[\bv_1\bq_1=b]=\PP[\bq_1=1]=\frac{\binom{n-1}{d-1}}{\binom{n}{d}}=\frac{d}{n}\le 1/2;
\]
and if $b=0$, then  
$$
\PP[\bq^{\intercal}\bv=b]=\PP[\bq_1=0]=\frac{\binom{n-1}{d}}{\binom{n}{d}}=\frac{n-d}{n}=1-\frac{d}{n}.
$$

\noindent
\emph{Case 2a: $|\supp(\bv)|\ge 2$ and $\bv\notin 1_{\supp(v)}\cdot \ZZ_p$.} W.l.o.g.\ we assume that $\bv_1, \bv_2\neq 0$; we further assume that $\bv_1\neq \bv_2$. Then, for $b\in\ZZ_p$ we have
\begin{align}
\PP[\bq^{\intercal}\bv\neq b]& = \Pr{\bq_1\bv_1+\bq_2\bv_2\neq b-\sum_{i=3}^{n} \bq_i\bv_i} \nonumber \\ 
& =\sum_{c\in\ZZ_p} \Pr{\bq_1\bv_1+\bq_2\bv_2\neq c \bigm\mid b-\sum_{i=3}^{n} \bq_i\bv_i=c} \Pr{b-\sum_{i=3}^{n} \bq_i\bv_i=c}.\label{eq:LO-light}
\end{align}

Let $c$ be given. If $c \notin\{0,\bv_1,\bv_2,\bv_1+\bv_2\}$, then $\Pr{\bq_1\bv_1+\bq_2\bv_2\neq c } = 1$. 

If $c \in\{0,\bv_1,\bv_2,\bv_1+\bv_2\}$, then we observe the following. 
\begin{itemize}
	\item If $c \in \{0,\bv_1+\bv_2\}$, then $c=\bq_1\bv_1+\bq_2\bv_2$ for 
	\emph{some} $q_1=q_2\in\{0,1\}$ which in turn implies that $c\neq\bq_1\bv_1+
	\bq_2\bv_2$ whenever $\bq_1 \neq \bq_2$ so that 
	\[
	\Pr{\bq_1\bv_1+\bq_2\bv_2\neq c \bigm\mid b-\sum_{i=3}^{n} \bq_i\bv_i=c}\ge 
	\Pr{\bq_1\neq \bq_2 \bigm\mid b-\sum_{i=3}^{n} \bq_i\bv_i=c}
	\]
	in particular holds. 
 	
	\item In the remaining case that $c \in \{\bv_1,\bv_2\}$, we may write that  
	\[
	\Pr{\bq_1\bv_1+\bq_2\bv_2\neq \bv_1 \bigm\mid b-\sum_{i=3}^{n} \bq_i\bv_i=\bv_1}= 
	\Pr{(\bq_1,\bq_2)\neq (1,0) \bigm\mid b-\sum_{i=3}^{n} \bq_i\bv_i=\bv_1}
    \]
    and 
	\[
    \Pr{\bq_1\bv_1+\bq_2\bv_2\neq \bv_2 \bigm\mid b-\sum_{i=3}^{n} \bq_i\bv_i=\bv_2}= 	   
    \Pr{(\bq_1,\bq_2)\neq (0,1) \bigm\mid b-\sum_{i=3}^{n} \bq_i\bv_i=\bv_2}.
	\]
\end{itemize}
The following lower bound on $\PP[\bq^{\intercal}\bv\neq b]$ now holds, owing to~\eqref{eq:LO-light},
\begin{align*}
\PP[\bq^{\intercal}\bv\neq b]\ge \sum_{c\in\{0,\bv_1+\bv_2\}}& \Pr{\bq_1\neq \bq_2 \bigm\mid b-\sum_{i=3}^{n} \bq_i\bv_i=c}\Pr{b-\sum_{i=3}^{n} \bq_i\bv_i=c}\\
& +\Pr{(\bq_1,\bq_2)\neq (1,0) \bigm\mid b-\sum_{i=3}^{n} \bq_i\bv_i=\bv_1}\Pr{b-\sum_{i=3}^{n} \bq_i\bv_i=\bv_1}\\
& +\Pr{(\bq_1,\bq_2)\neq (0,1) \bigm\mid b-\sum_{i=3}^{n} \bq_i\bv_i=\bv_2}\Pr{b-\sum_{i=3}^{n} \bq_i\bv_i=\bv_2}\\
& +\sum_{c\notin\{0,v_1,v_2,v_1+v_2\}} \Pr{b-\sum_{i=3}^{n} \bq_i\bv_i=c}
\end{align*}

Observe that 
$$
\Pr{(\bq_1,\bq_2)\neq (1,0) \bigm\mid b-\sum_{i=3}^{n} \bq_i\bv_i=\bv_1}=\Pr{(\bq_1,\bq_2)\neq (0,1) \bigm\mid b-\sum_{i=3}^{n} \bq_i\bv_i=\bv_1}.
$$ 
To see this, note that 
$$
\Pr{(\bq_1,\bq_2) =(0,1) \bigm\mid b-\sum_{i=3}^{n} \bq_i\bv_i=\bv_1}=\Pr{(\bq_1,\bq_2)= (1,0) \bigm\mid b-\sum_{i=3}^{n} \bq_i\bv_i=\bv_1}
$$ 
as well as
\[
\Pr{(\bq_1,\bq_2)\neq (1,0) \bigm\mid b-\sum_{i=3}^{n} \bq_i\bv_i=\bv_1}=\Pr{(\bq_1,\bq_2) =(0,1) \bigm\mid b-\sum_{i=3}^{n} \bq_i\bv_i=\bv_1}+\Pr{\bq_1=\bq_2 \bigm\mid b-\sum_{i=3}^{n} \bq_i\bv_i=\bv_1}
\]
and 
\[
\Pr{(\bq_1,\bq_2)\neq (0,1) \bigm\mid b-\sum_{i=3}^{n} \bq_i\bv_i=\bv_1}=\Pr{(\bq_1,\bq_2) =(1,0) \bigm\mid b-\sum_{i=3}^{n} \bq_i\bv_i=\bv_1}+\Pr{\bq_1=\bq_2 \bigm\mid b-\sum_{i=3}^{n} \bq_i\bv_i=\bv_1}.
\]

The following crude lower bound for $\PP[\bq^{\intercal}\bv\neq b]$ is then reached,
\begin{align*}
\PP[\bq^{\intercal}\bv\neq b]& \ge \sum_{c\in \ZZ_p} \Pr{(\bq_1,\bq_2)= (1,0) \bigm\mid b-\sum_{i=3}^{n} \bq_i\bv_i=c}\Pr{b-\sum_{i=3}^{n} \bq_i\bv_i=c}\\
&=\Pr{(\bq_1,\bq_2)= (1,0)}\\
&=\frac{\binom{n-2}{d-1}}{\binom{n}{d}}=\frac{d(n-d)}{n(n-1)}>\frac{d}{2n}.
\end{align*}

\noindent
\emph{Case 2b: $|\supp(\bv)|\ge 2$ and $\bv\in \bone_{\supp(\bv)}\cdot \ZZ_p$.} W.l.o.g.\ we further assume that $\bv_1, \bv_2\neq 0$ and that $\bv=\bone_S$ for some set $S\supseteq\{1,2\}$. 
Let $b\in\ZZ_p$ and set $s:=|S|$. Then, $\PP[\bq^{\intercal}\bv=b]=\Pr{\sum_{i\in S} \bq_i=b}$ so that the problem reduces to counting the number of sets $T\in\binom{[n]}{d}$ satisfying $|T\cap S|\equiv b\bmod p$. Writing $|T \cap S| = ip+b$, note that the inequalities $ip+b \leq d$ and $ip+b \leq s$ arise as 
$|T| =d$ and $|T \cap S| \leq s$, respectively. Hence, we may write 
\begin{equation}\label{eq:LO-bad}
\PP[\bq^{\intercal}\bv=b]=\Pr{\sum_{i\in S} \bq_i=b}=\frac{\sum_{i=0}^{\ell} \binom{s}{ip+b}\binom{n-s}{d-ip-b}}{\binom{n}{d}},
\end{equation}
where 
$
\ell=\min\left\{\lfloor (s-b)/p\rfloor, \lfloor(d-b)/p\rfloor\right\}
$
defines the upper bound that the parameter $i$ in $ip+b$ can attain. 

Surely, 
$$
\PP[\bq^{\intercal}\bv\neq b] \geq \PP[\bq^{\intercal}\bv =  b-1] + \PP[\bq^{\intercal}\bv =  b+1];
$$
hence we attain
\begin{equation}\label{eq:LO-good}
\PP[\bq^{\intercal}\bv\neq b] \geq \frac{\sum_{i=0}^{\ell} \left(\binom{s}{ip+b-1}\binom{n-s}{d-ip-b+1}+\binom{s}{ip+b+1}\binom{n-s}{d-ip-b-1}\right)}{\binom{n}{d}}.
\end{equation}

Let 
\[
a_i:=\binom{s}{ip+b-1}\binom{n-s}{d-ip-b+1}+\binom{s}{ip+b+1}\binom{n-s}{d-ip-b-1}
\]
denote the numerator of~\eqref{eq:LO-bad}; and let 
 \[
 c_i:= \binom{s}{ip+b}\binom{n-s}{d-ip-b}
 \]
denote the numerator of~\eqref{eq:LO-good}. In what follows we prove that $a_i \geq \frac{d}{n}c_i$ for every $i \in [\ell]$. This in turn yields that $\PP[\bq^{\intercal}\bv \neq b] \geq \frac{d}{n} \PP[\bq^{\intercal}\bv =b]$. The latter coupled with the triviality that $\PP[\bq^{\intercal}\bv \neq b] + \PP[\bq^{\intercal}\bv = b] = 1$ implies that $\PP[\bq^{\intercal}\bv \neq b] \geq \frac{d}{n+d} \geq \frac{d}{2n}$ so that~\eqref{eq:=b} follows completing the proof in this case.

\smallskip
It remains to prove that $a_i \geq \frac{d}{2n}c_i$ for every $i \in [\ell]$. 
We start by setting $t:=ip+b$ and rewriting the summand $a_i$ as
\begin{equation}\label{eq:LO-compare}
\binom{s}{t-1}\binom{n-s}{d-t+1}+\binom{s}{t+1}\binom{n-s}{d-t-1}=\binom{s}{t}\binom{n-s}{d-t}\left(\frac{t}{s-t+1}\cdot\frac{n-s-d+t}{d-t+1}+
\frac{s-t}{t+1}\cdot\frac{d-t}{n-s-d+t+1}\right).
\end{equation}
Owing to the assumption appearing in the premise that $\bv \notin \spn(\bone)$, we have $t\le s$, $t\le d$ and that $s<n$.

Dealing with the residual cases below, assume, first, that $s>t\ge 1$, $d>t$ and $n-s> d-t$ (thus: $n-d> s-t$). Then, subject to these assumptions we may write  
\begin{align*}
\frac{t}{s-t+1}\cdot\frac{n-s-d+t}{d-t+1}+
\frac{s-t}{t+1}\cdot\frac{d-t}{n-s-d+t+1}&\ge 
\frac{t}{2(s-t)}\cdot\frac{(n-d)-(s-t)}{2(d-t)}+
\frac{s-t}{2t}\cdot\frac{d-t}{2((n-d)-(s-t))}\\
& \ge \frac{1}{4}\left(\frac{t}{s-t}\cdot\frac{(n-d)-(s-t)}{d-t}+
\frac{s-t}{t}\cdot\frac{d-t}{(n-d)-(s-t)}\right)\\ 
& =
\frac{1}{4}\left(x+
\frac{1}{x}\right),
\end{align*}
where $x\colon=\frac{t}{s-t}\cdot\frac{(n-d)-(s-t)}{d-t}>0$. As $x+\frac{1}{x}\ge 2$ always holds for $x > 0$, we attain the following lower bound on~\eqref{eq:LO-compare},
\[
\binom{s}{t-1}\binom{n-s}{d-t+1}+\binom{s}{t+1}\binom{n-s}{d-t-1}\ge \frac{1}{2}\binom{s}{t}\binom{n-s}{d-t} = \frac{c_i}{2}.
\]

Next, we contend with the residual cases left to consider for our estimate of 
$$
\frac{t}{s-t+1}\cdot\frac{n-s-d+t}{d-t+1}+\frac{s-t}{t+1}\cdot\frac{d-t}{n-s-d+t+1},
$$
appearing on the r.h.s. of~\eqref{eq:LO-compare}, to be complete. As these `boundary' cases are somewhat more docile, so to speak, compared to our `primary' case, these all fit nicely within the following short list. 
\begin{itemize}

\item If $n-s\le d-t-1$, then $a_i=\binom{s}{t+1}\binom{n-s}{d-t-1}$, whereas $c_i=0$, thus $a_i\ge c_i$ holds.
\item If $n-s= d-t$ (so that $n-d=s-t$), $t<d$,  and $s>t$ (and recalling that $d\le n/2$), then $c_i=\binom{s}{t}$, whereas 
$$
a_i=\binom{s}{t+1}(n-s)= 
\binom{s}{t}\frac{(s-t)(n-s)}{t+1}= \binom{s}{t}\frac{(n-d)(n-s)}{t+1}\ge \binom{s}{t}\frac{(n-d)(n-s)}{d}\ge  \binom{s}{t}(n-s)\ge \binom{s}{t}.
$$

\item If $t=s$, then $c_i=\binom{n-s}{d-s}$ whereas 
$$
a_i=s\binom{n-s}{d-s+1}=s\frac{n-d}{d-s+1}\binom{n-s}{d-s}>s\binom{n-s}{d-s}\ge c_i.
$$

\item If $t=d$, then 
$$
c_i=\binom{s}{d}=\frac{s-d+1}{d}\binom{s}{d-1}< \frac{n}{d} \binom{s}{d-1}
$$ 
whereas 
$$
a_i=\binom{s}{d-1}(n-s)\ge \binom{s}{d-1}\ge \frac{d}{n} c_i.
$$

\item If $t=0$ and $s\le n-d$, then $c_i=\binom{n-s}{d}$ whereas 
$$
a_i=s\binom{n-s}{d-1}=s\frac{d}{n-d-s+1}\binom{n-s}{d}>\frac{2d}{n}\binom{n-s}{d}=\frac{2d}{n}c_i
$$ 
(since $s\ge 2$).

\item If $t=0$ and $s> n-d$, then $c_i=\binom{s}{0}\binom{n-s}{d}=0$, so that $a_i\ge c_i$.
\end{itemize}

Observe that throughout the cases considered, $\PP[\bq^{\intercal}\bv\neq b]\ge \frac{d}{2n}$ is always maintained so that~\eqref{eq:=b} holds with $c'=1/2$.
\end{proof}

\end{document}